\documentclass[12pt]{article}
\usepackage{epsfig}
\usepackage{amsmath,amsthm,amsfonts,amssymb,amscd,cite}

= 0.0in \headheight = 0.0in \topmargin = 0.3in
\def\Dj{\hbox{D\kern-.73em\raise.30ex\hbox{-}
\raise-.30ex\hbox{}}}
\def\dj{\hbox{d\kern-.33em\raise.80ex\hbox{-}
\raise-.80ex\hbox{\kern-.40em}}}
\newtheorem{theorem}{Theorem}[section]

\newtheorem{lemma}[theorem]{Lemma}
\newtheorem{corollary}[theorem]{Corollary}
\newtheorem{remark}[theorem]{Remark}

\input{tcilatex}
\begin{document}

\title{\textbf{Note on the Sum of Powers of Signless Laplacian Eigenvalues
of Graphs}}
\author{\c{S}. Burcu Bozkurt Alt\i nda\u{g}\thanks{%
Corresponding author} and Durmu\c{s} Bozkurt \\
\textit{Department of Mathematics, Science Faculty, }\\
\textit{Sel\c{c}uk University, 42075, Campus, Konya, Turkey}\\
\textit{srf\_burcu\_bozkurt@hotmail.com, dbozkurt@selcuk.edu.tr}}
\maketitle

\begin{abstract}
For a simple graph $G$ and a real number $\alpha $ $\left( \alpha \neq
0,1\right) $ the graph invariant $s_{\alpha }\left( G\right) $ is equal to
the sum of powers of signless Laplacian eigenvalues of $G$. In this note, we
present some new bounds on $s_{\alpha }\left( G\right) $. As a result of
these bounds, we also give some results on incidence energy.
\end{abstract}

\section{\protect\bigskip\ Introduction}

Let $G$ be a simple graph with $n$\ vertices and $m$ edges. Let $V\left(
G\right) =\left\{ v_{1},v_{2},\ldots ,v_{n}\right\} $ be the set of vertices
of $G$. For $v_{i}\in V\left( G\right) $, the degree of the vertex $v_{i}$,
denoted by $d_{i},$ is equal to the number of vertices adjacent to $v_{i}.$
Throughout this paper, the maximum, the second maximum and the minimum \
vertex degrees of $G$ will be denoted by $\Delta _{1}$, $\Delta _{2}$ and $%
\delta $, respectively.

Let $A\left( G\right) $ be the $\left( 0,1\right) $-adjacency matrix of a
graph $G$. The eigenvalues of $G$ are the eigenvalues of $A\left( G\right) $ 
\cite{6} and denoted by $\lambda _{1}\geq \lambda _{2}\geq \cdots \geq
\lambda _{n}.$ Then the energy of a graph $G$ is defined by \cite{17}%
\begin{equation*}
E=E\left( G\right) =\dsum\limits_{i=1}^{n}\left\vert \lambda _{i}\right\vert
.
\end{equation*}%
There is an extensive literature on this topic. For more details see \cite%
{18,28} and the references cited therein.

The concept of graph energy was extended to energy of any matrix in the
following manner \cite{36}. Recall that the singular values of any (real)
matrix $M$ are equal to the square roots of the eigenvalues of $MM^{T}$,
where $M^{T}$ is the transpose of $M$. Then the energy of the matrix $M$ is
defined as the sum of its singular values. Clearly, $E\left( A\left(
G\right) \right) =E\left( G\right) $.

Let $D\left( G\right) $ be the diagonal matrix of vertex degrees of $G$.
Then the Laplacian matrix of $G$ is $L\left( G\right) =D\left( G\right)
-A\left( G\right) $ and the signless Laplacian matrix of $G$ is $Q\left(
G\right) =D\left( G\right) +A\left( G\right) $. As well known in spectral
graph theory, both $L\left( G\right) $ and $Q\left( G\right) $ are real
symmetric and positive semidefinite matrices, so their eigenvalues are
non-negative real numbers. Let $\mu _{1}\geq \mu _{2}\geq \cdots \geq \mu
_{n}=0$ be the eigenvalues of $L\left( G\right) $ and let $q_{1}\geq
q_{2}\geq \cdots \geq q_{n}$ be the eigenvalues of $Q\left( G\right) $.
These eigenvalues are called Laplacian and signless Laplacian eigenvalues of 
$G$, respectively. For details on Laplacian and signless Laplacian
eigenvalues, see \cite{7,8,9,10,33,34}.

The incidence matrix $I\left( G\right) $ of a graph $G$ with the vertex set $%
V\left( G\right) =\left\{ v_{1},v_{2},\ldots ,v_{n}\right\} $ and edge set $%
E\left( G\right) =\left\{ e_{1},e_{2},\ldots ,e_{m}\right\} $ is the matrix
whose $\left( i,j\right) $-entry is $1$ if the vertex $v_{i}$ is incident
with the edge $e_{j}$, and \ is $0$ otherwise. In \cite{24}, Jooyandeh et
al. motivated the idea in \cite{36} and defined the incidence energy of $G$,
denoted by $IE\left( G\right) $, as the sum of singular values of $I\left(
G\right) $. Since $Q\left( G\right) =I\left( G\right) I\left( G\right) ^{T}$%
, it was later proved that \cite{20} 
\begin{equation*}
IE=IE\left( G\right) =\dsum\limits_{i=1}^{n}\sqrt{q_{i}}
\end{equation*}%
For the basic properties of $IE$ involving also its lower and upper bounds,
see \cite{3,4,13,20,21,24,32,38,42,43}.

In \cite{30} Liu and Lu introduced a new graph invariant based on Laplacian
eigenvalues 
\begin{equation*}
LEL=LEL\left( G\right) =\dsum\limits_{i=1}^{n-1}\sqrt{\mu _{i}}
\end{equation*}%
and called it Laplacian energy like invariant. At first it was considered
that \cite{30} $LEL$ shares similar properties with Laplacian energy \cite%
{22}. Then it was shown that it is much more similar to the ordinary graph
energy \cite{23}. For survey and details on $LEL$, see \cite{29}.

For a graph $G$ with $n$ vertices and a real number $\alpha $, to avoid
trivialities it may be required that $\alpha \neq 0,1$, the sum of the $%
\alpha $th powers of the non-zero Laplacian eigenvalues is defined as \cite%
{41}%
\begin{equation*}
\sigma _{\alpha }=\sigma _{\alpha }\left( G\right)
=\dsum\limits_{i=1}^{n-1}\mu _{i}^{\alpha }.
\end{equation*}%
The cases $\alpha =0$ and $\alpha =1$ are trivial as $\sigma _{0}=n-1$ and $%
\sigma _{1}=2m$, where $m$ is the number of edges of $G$. Note that $\sigma
_{1/2}$ is equal to $LEL$. It is worth noting that $n\sigma _{-1}$ is also
equal to the Kirchhoff index of $G$ (one can refer to the papers \cite%
{2,19,37} for its definition and extensive applications in the theory of
electric circuits, probabilistic theory and chemistry). Recently, various
properties and the estimates of $\sigma _{\alpha }$ have been well studied
in the literature. For details, see \cite{14,31,39,41,43}.

Motivating the definitions of $IE,$ $LEL$ and $\sigma _{\alpha }$, Akbari et
al. \cite{1} introduced the sum of the $\alpha $th powers of the signless
Laplacian eigenvalues of $G$ as 
\begin{equation*}
s_{\alpha }=s_{\alpha }\left( G\right) =\dsum\limits_{i=1}^{n}q_{i}^{\alpha }
\end{equation*}%
and they also gave some relations between $\sigma _{\alpha }$ and $s_{\alpha
}$. In this sum, the cases $\alpha =0$ and $\alpha =1$ are trivial as $%
s_{0}=n$ and $s_{1}=2m$. Note that $s_{1/2}$ is equal to the incidence
energy $IE$. Note further that Laplacian eigenvalues and signless Laplacian
eigenvalues of bipartite graphs coincide \cite{7,33,34}. Therefore, for
bipartite graphs $\sigma _{\alpha }$ is equal to $s_{\alpha }$ \cite{3} and $%
LEL$ is equal to $IE$ \cite{20}. Recently some properties and the lower and
upper bounds of $s_{\alpha }$ have been established in \cite{1,3,27,32}.

In this paper, we obtain some new bounds on $s_{\alpha }$ of bipartite
graphs which improve the some bounds in \cite{14}. In addition to this, we
extend these bounds to non-bipartite graphs. As a result of these bounds, we
also present some results on incidence energy.

\section{Lemmas}

\bigskip Let $t=t\left( G\right) $ denotes the number of spanning trees of $%
G $. Let $\overline{G}$ be the complement of $G$ and let $G_{1}\times G_{2}$
be the Cartesian product of the graphs $G_{1}$ and $G_{2}$ \cite{6}. Now, we
give two auxiliary quantities for a graph $G$ as%
\begin{equation}
t_{1}=t_{1}\left( G\right) =\frac{2t\left( G\times K_{2}\right) }{t\left(
G\right) }\text{ and }T=T\left( G\right) =\frac{1}{2}\left[ \Delta
_{1}+\delta +\sqrt{\left( \Delta _{1}-\delta \right) ^{2}+4\Delta _{1}}%
\right]
\end{equation}%
where $\Delta _{1}$ and $\delta $ are the maximum and the minimum vertex
degrees of $G$, respectively.

\begin{lemma}
\cite{25} Let $G$ be a graph with $n$ vertices and $m$ edges. Then%
\begin{equation}
\dsum\limits_{i=1}^{n}d_{i}^{2}\leq m\left( \frac{2m}{n-1}+n-2\right) .
\end{equation}%
Moreover, if $G$ is connected, then the equality holds in (2) if and only $G$
is either a star $K_{1,n-1}$ or a complete graph $K_{n}.$
\end{lemma}

\begin{lemma}
\cite{7,33,34} The spectra of $L\left( G\right) $ and $Q\left( G\right) $
coincide if and only if the graph $G$ is bipartite.
\end{lemma}

\begin{lemma}
\cite{9} If $G$ is a connected bipartite graph of order $n$, then $%
\dprod\nolimits_{i=1}^{n-1}q_{i}=\dprod\nolimits_{i=1}^{n-1}\mu
_{i}=nt\left( G\right) $. If $G$ is a connected non-bipartite graph of order 
$n$, then $\dprod\nolimits_{i=1}^{n}q_{i}=t_{1}\left( G\right) $.
\end{lemma}

\begin{lemma}
\cite{5,33} Let $G$ be a connected graph with $n\geq 3$ vertices and maximum
vertex degree $\Delta _{1}$. Then%
\begin{equation*}
q_{1}\geq T\geq \Delta _{1}+1
\end{equation*}%
with either equalities if and only if $G$ is a star graph $K_{1,n-1}$.
\end{lemma}

\begin{lemma}
\cite{11} Let $G$ be a graph with second maximum vertex degree $\Delta _{2}$%
. Then%
\begin{equation*}
q_{2}\geq \Delta _{2}-1.
\end{equation*}%
If $q_{2}=\Delta _{2}-1$, then the maximum and the second maximum vertex
degrees are adjacent and $\Delta _{1}=\Delta _{2}$.
\end{lemma}

\begin{lemma}
\cite{11} Let $G$ be a connected graph with $n$ vertices and minimum vertex
degree $\delta $. Then%
\begin{equation*}
q_{n}<\delta .
\end{equation*}
\end{lemma}

\begin{lemma}
\cite{8} Let $G$ be a connected graph with diameter $d\left( G\right) $. If $%
G$ has exactly $k$ distinct signless Laplacian eigenvalues, then $d\left(
G\right) +1\leq k.$
\end{lemma}

\begin{lemma}
\cite{26} Let $G$ be a connected graph with $n\geq 3$ vertices and second
maximum vertex degree $\Delta _{2}$. Then%
\begin{equation*}
\mu _{2}\geq \Delta _{2}
\end{equation*}%
with equality if $G$ is a complete bipartite graph $K_{p,q}$ or a tree with
degree sequence $\pi \left( T_{n}\right) =\left( n/2,n/2,1,1,\ldots
,1\right) $, where $n\geq 4$ is even.
\end{lemma}

\begin{lemma}
\cite{15} Let $G$ be a graph with $n$ vertices, different from $K_{n}$ and
let $\delta $ be the minimum vertex degree of $G$. Then%
\begin{equation*}
\mu _{n-1}\leq \delta
\end{equation*}
\end{lemma}

\begin{lemma}
\cite{12,41} Let $G$ be a simple graph with $n$ vertices. Then $\mu _{1}=\mu
_{2}=\cdots =\mu _{n-1}$ if and only if $G\cong K_{n}$ or $G\cong \overline{K%
}_{n}$.
\end{lemma}

\begin{lemma}
\cite{16} For $a_{1},a_{2},\ldots ,a_{n}\geq 0$ and $p_{1},p_{2},\ldots
,p_{n}\geq 0$ such that $\sum\nolimits_{i=1}^{n}p_{i}=1$%
\begin{equation}
\sum\limits_{i=1}^{n}p_{i}a_{i}-\prod\limits_{i=1}^{n}a_{i}^{p_{i}}\geq
n\lambda \left( \frac{1}{n}\sum\limits_{i=1}^{n}a_{i}-\prod%
\limits_{i=1}^{n}a_{i}^{1/n}\right)
\end{equation}%
where $\lambda =\min \left\{ p_{1},p_{2},\ldots ,p_{n}\right\} $. Moreover,
equality holds in (3) if and only if $a_{1}=a_{2}=\cdots =a_{n}$.
\end{lemma}

\begin{lemma}
\cite{35} Let $a_{i}>0$, $i=1,2,\ldots ,p$ be the $p$ real numbers. Then%
\begin{equation*}
p\left( A_{p}-G_{p}\right) \geq \left( p-1\right) \left(
A_{p-1}-G_{p-1}\right) ,
\end{equation*}%
where%
\begin{equation*}
A_{p}=\frac{\sum\nolimits_{i=1}^{p}a_{i}}{p}\text{ and }G_{p}=\left(
\prod\limits_{i=1}^{p}a_{i}\right) ^{1/p}.
\end{equation*}
\end{lemma}

\section{Main Results}

In this section, we give the main results of the paper. First, we need the
following lemma. For a graph $G$ with signless Laplacian eigenvalues $%
q_{1}\geq q_{2}\geq \cdots \geq q_{n}$, let 
\begin{equation*}
M_{k}=M_{k}\left( G\right) =\dsum\nolimits_{i=1}^{k}q_{i}
\end{equation*}%
for $1\leq k\leq n-1$. Then, we have:

\begin{lemma}
Let $G$ be a connected graph with $n$ vertices and $m$ edges.

i) If $G$ is bipartite, then for $1\leq k\leq n-2$%
\begin{equation}
M_{k}\left( G\right) \leq \frac{2mk+\sqrt{mk\left( n-k-1\right) \left(
n^{2}-n-2m\right) }}{n-1}
\end{equation}%
with equality holding in (4) if and only if $G$ is either a star $K_{1,n-1}$
or a complete graph $K_{n}$ when $k=1$ and $G$ is a complete graph $K_{n}$
when $2\leq k\leq n-2.$

ii) If $G$ is non-bipartite, then for $1\leq k\leq n-1$ 
\begin{equation}
M_{k}\left( G\right) \leq \frac{2mk+\sqrt{mk\left( n-k\right) \left( n^{2}+%
\frac{2mn}{n-1}-4m\right) }}{n}
\end{equation}%
with equality holding in (5) if and only if $G\cong K_{n}$ when $k=1$.
\end{lemma}

\begin{proof}
The inequality (4) was established in \cite{40}. So we omit its proof here.
Now we only prove the inequality (5). Let $M_{k}=M_{k}\left( G\right) $. It
is clear that \cite{7}%
\begin{equation*}
q_{1}+q_{2}+\cdots +q_{n}=2m 
\end{equation*}%
and%
\begin{equation*}
q_{1}^{2}+q_{2}^{2}+\cdots +q_{n}^{2}=2m+\dsum\nolimits_{i=1}^{n}d_{i}^{2} 
\end{equation*}%
Then, using Cauchy-Schwarz inequality, we get%
\begin{eqnarray*}
\left( 2m-M_{k}\right) ^{2} &=&\left( q_{k+1}+\cdots +q_{n}\right) ^{2} \\
&\leq &\left( n-k\right) \left( q_{k+1}^{2}+\cdots +q_{n}^{2}\right) \\
&=&\left( n-k\right) \left( 2m+\dsum\nolimits_{i=1}^{n}d_{i}^{2}-\left(
q_{1}^{2}+\cdots +q_{k}^{2}\right) \right) \\
&\leq &\left( n-k\right) \left( 2m+\dsum\nolimits_{i=1}^{n}d_{i}^{2}-\frac{1%
}{k}M_{k}^{2}\right) .
\end{eqnarray*}%
Therefore%
\begin{equation}
M_{k}\leq \left\{ 2mk+\left[ k\left( n-k\right) \left( n\left(
2m+\dsum\nolimits_{i=1}^{n}d_{i}^{2}\right) -4m^{2}\right) \right]
^{1/2}\right\} /n.
\end{equation}%
From the inequality (6) and Lemma 2.1, the inequality (5) holds. Now we
suppose that the equality holds in (5). Then, by Cauchy-Schwarz inequality
we have $q_{1}=\cdots =q_{k}$ and $q_{k+1}=\cdots =q_{n}$. Since $G$ is
connected non-bipartite graph, by Lemma 2.1 and Lemma 2.7, we conclude that $%
G\cong K_{n}$ when $k=1.$
\end{proof}

The following result can be found in \cite{14}.

\begin{theorem}
\cite{14} Let $G$ be a bipartite graph with $n\geq 2$ vertices, $m$ edges
and positive integer $k$ $\left( 1\leq k\leq n-2\right) $.

(i) If $0<\alpha <1$, then%
\begin{equation}
s_{\alpha }\left( G\right) =\sigma _{\alpha }\left( G\right) \leq
k^{1-\alpha }\left( \frac{2mk}{n-1}\right) ^{\alpha }+\left( n-k-1\right)
^{1-\alpha }\left( 2m-\frac{2mk}{n-1}\right) ^{\alpha }
\end{equation}%
with equality holding in (7) if and only if $G\cong K_{n}$ or $G\cong 
\overline{K}_{n}$.

(ii) If $\alpha >1$, then 
\begin{equation}
s_{\alpha }\left( G\right) =\sigma _{\alpha }\left( G\right) \geq
k^{1-\alpha }\left( \frac{2mk}{n-1}\right) ^{\alpha }+\left( n-k-1\right)
^{1-\alpha }\left( 2m-\frac{2mk}{n-1}\right) ^{\alpha }
\end{equation}%
with equality holding in (8) if and only if $G\cong K_{n}$ or $G\cong 
\overline{K}_{n}$.

(iii) If $G$ is connected and $\alpha <0$, then 
\begin{equation}
\begin{array}{c}
s_{\alpha }\left( G\right) =\sigma _{\alpha }\left( G\right) \leq
\min\limits_{1\leq k\leq n-2}\left\{ k^{1-\alpha }\left[ \frac{2mk+\sqrt{%
mk\left( n-k-1\right) \left( n^{2}-n-2m\right) }}{n-1}\right] ^{\alpha
}\right. \\ 
\left. +\left( n-k-1\right) ^{1-\alpha }\left[ \frac{2mk\left( n-k-1\right) -%
\sqrt{mk\left( n-k-1\right) \left( n^{2}-n-2m\right) }}{n-1}\right] ^{\alpha
}\right\}%
\end{array}%
\end{equation}%
with equality holding in (9) if and only if $G\cong K_{1,n-1}$ $\left(
k=1\right) $ and $G\cong K_{n}$ $\left( 2\leq k\leq n-2\right) .$
\end{theorem}

We now extend the above result to non-bipartite graphs.

\begin{theorem}
Let $G$ be a non-bipartite graph with $n\geq 2$ vertices, $m$ edges and
positive integer $k$ $\left( 1\leq k\leq n-1\right) $.

(i) If $0<\alpha <1$, then%
\begin{equation}
s_{\alpha }\left( G\right) \leq k^{1-\alpha }\left( \frac{2mk}{n}\right)
^{\alpha }+\left( n-k\right) ^{1-\alpha }\left( 2m-\frac{2mk}{n}\right)
^{\alpha }
\end{equation}%
with equality holding in (10) if and only if $G\cong \overline{K}_{n}$.

(ii) If $\alpha >1$, then 
\begin{equation}
s_{\alpha }\left( G\right) \geq k^{1-\alpha }\left( \frac{2mk}{n}\right)
^{\alpha }+\left( n-k\right) ^{1-\alpha }\left( 2m-\frac{2mk}{n}\right)
^{\alpha }
\end{equation}%
with equality holding in (11) if and only if $G\cong \overline{K}_{n}$.

(iii) If $G$ is connected and $\alpha <0$, then%
\begin{equation}
\begin{array}{c}
s_{\alpha }\left( G\right) \leq \min\limits_{1\leq k\leq n-1}\left\{
k^{1-\alpha }\left[ \frac{2mk+\sqrt{mk\left( n-k\right) \left( n^{2}+\frac{%
2mn}{n-1}-4m\right) }}{n}\right] ^{\alpha }\right. \\ 
\left. +\left( n-k\right) ^{1-\alpha }\left[ \frac{2m\left( n-k\right) -%
\sqrt{mk\left( n-k\right) \left( n^{2}+\frac{2mn}{n-1}-4m\right) }}{n}\right]
^{\alpha }\right\}%
\end{array}%
\end{equation}%
with equality holding in (12) if and only if $G\cong K_{n}$ when $k=1$.
\end{theorem}

\begin{proof}
Using power mean inequality, we get%
\begin{equation}
\sum\limits_{i=1}^{k}q_{i}^{\alpha }\leq k^{1-\alpha }\left(
\sum\limits_{i=1}^{k}q_{i}\right) ^{\alpha }\text{, as }0<\alpha <1
\end{equation}%
with equality holding in (13) if and only if $q_{1}=q_{2}=\cdots =q_{k}$.

Considering the above manner, we also get%
\begin{equation}
\sum\limits_{i=k+1}^{n}q_{i}^{\alpha }\leq \left( n-k\right) ^{1-\alpha
}\left( 2m-\sum\limits_{i=1}^{k}q_{i}\right) ^{\alpha }\text{, as }%
\dsum\nolimits_{i=1}^{n}q_{i}=2m\text{ \cite{7},}
\end{equation}%
with equality holding in (14) if and only if $q_{k+1}=q_{k+2}=\cdots =q_{n}$%
. Since $q_{1}\geq $ $q_{2}\geq \cdots \geq q_{n}$, we have 
\begin{equation*}
\frac{\sum\nolimits_{i=1}^{k}q_{i}}{k}\geq \frac{\sum%
\nolimits_{i=k+1}^{n}q_{i}}{n-k}=\frac{2m-\sum\nolimits_{i=1}^{k}q_{i}}{n-k}%
. 
\end{equation*}%
Therefore, we get%
\begin{equation}
\sum\nolimits_{i=1}^{k}q_{i}\geq \frac{2mk}{n}.
\end{equation}%
By \ Eqs. (13) and (14), we obtain%
\begin{eqnarray*}
s_{\alpha }\left( G\right) &=&\sum\limits_{i=1}^{n}q_{i}^{\alpha
}=\sum\limits_{i=1}^{k}q_{i}^{\alpha }+\sum\limits_{i=k+1}^{n}q_{i}^{\alpha }
\\
&\leq &k^{1-\alpha }\left( \sum\limits_{i=1}^{k}q_{i}\right) ^{\alpha
}+\left( n-k\right) ^{1-\alpha }\left( 2m-\sum\limits_{i=1}^{k}q_{i}\right)
^{\alpha }.
\end{eqnarray*}%
Now consider the following function%
\begin{equation*}
f\left( x\right) =k^{1-\alpha }x^{\alpha }+\left( n-k\right) ^{1-\alpha
}\left( 2m-x\right) ^{\alpha } 
\end{equation*}%
for $x\geq \frac{2mk}{n}$. Then it is easy to see that 
\begin{equation*}
f^{\prime }\left( x\right) =\alpha \left[ \left( \frac{x}{k}\right) ^{\alpha
-1}-\left( \frac{2m-x}{n-k}\right) ^{\alpha -1}\right] \leq 0\text{, as }%
0<\alpha <1. 
\end{equation*}%
Thus, by (15), we get%
\begin{equation*}
f\left( x\right) \leq f\left( \frac{2mk}{n}\right) =k^{1-\alpha }\left( 
\frac{2mk}{n}\right) ^{\alpha }+\left( n-k\right) ^{1-\alpha }\left( 2m-%
\frac{2mk}{n}\right) ^{\alpha }. 
\end{equation*}%
Hence we get the the inequality (10). Now we suppose that the equality holds
in (10). Then, from (13) and (14) we have $q_{1}=q_{2}=\cdots =q_{k}$ and $%
q_{k+1}=q_{k+2}=\cdots =q_{n}$, respectively. Furthermore from (15), we have%
\begin{equation*}
\sum\limits_{i=1}^{k}q_{i}=\frac{2mk}{n}. 
\end{equation*}%
Therefore 
\begin{equation*}
q_{1}=q_{2}=\cdots =q_{n}=\frac{2m}{n}. 
\end{equation*}%
Then, we conclude that $G\cong \overline{K}_{n}$.

Conversely, one can easily show that the equality holds in (10) for the
complement of the complete graph $\overline{K}_{n}.$

(ii) Using power mean inequality, from (i), we obtain%
\begin{equation*}
s_{\alpha }\left( G\right) \geq k^{1-\alpha }\left(
\sum\limits_{i=1}^{k}q_{i}\right) ^{\alpha }+\left( n-k\right) ^{1-\alpha
}\left( 2m-\sum\limits_{i=1}^{k}q_{i}\right) ^{\alpha },\text{as }\alpha >1. 
\end{equation*}%
Note that $f\left( x\right) $ is increasing function for $x\geq \frac{2mk}{n}
$ as $\alpha >1$. Then, similar to the proof of (i), we get the inequality
(11). Furthermore, the equality holds in (11) if and only if $G\cong 
\overline{K}_{n}$.

(iii) From Lemma 3.1, we have%
\begin{equation*}
\sum\limits_{i=1}^{k}q_{i}\leq \frac{2mk+\sqrt{mk\left( n-k\right) \left(
n^{2}+\frac{2mn}{n-1}-4m\right) }}{n}. 
\end{equation*}%
As $\alpha <0$, from (i), we obtain that $f\left( x\right) $ is increasing
function for 
\begin{equation*}
\frac{2mk}{n}\leq x\leq \frac{1}{n}\left[ 2mk+\sqrt{mk\left( n-k\right)
\left( n^{2}+\frac{2mn}{n-1}-4m\right) }\right] . 
\end{equation*}%
Therefore%
\begin{eqnarray*}
f\left( x\right) &\leq &k^{1-\alpha }\left( \frac{2mk+\sqrt{mk\left(
n-k\right) \left( n^{2}+\frac{2mn}{n-1}-4m\right) }}{n}\right) ^{\alpha
}+\left( n-k\right) ^{1-\alpha } \\
&&\times \left( \frac{2m\left( n-k\right) -\sqrt{mk\left( n-k\right) \left(
n^{2}+\frac{2mn}{n-1}-4m\right) }}{n}\right) ^{\alpha }.
\end{eqnarray*}%
Hence the inequality (12) holds. Now we suppose that the equality holds in
(12). Therefore we get that 
\begin{equation*}
q_{1}=q_{2}=\cdots =q_{k}\text{, }q_{k+1}=q_{k+2}=\cdots =q_{n} 
\end{equation*}%
and 
\begin{equation*}
\sum\limits_{i=1}^{k}q_{i}=\frac{2mk+\sqrt{mk\left( n-k\right) \left( n^{2}+%
\frac{2mn}{n-1}-4m\right) }}{n}. 
\end{equation*}%
Then, from Lemma 3.1, we conclude that $G\cong K_{n}$ when $k=1$.

Conversely, let $G$ be isomorphic to the complete graph $K_{n}$ when $k=1$.
Thus%
\begin{eqnarray*}
&&k^{1-\alpha }\left[ \frac{2mk+\sqrt{mk\left( n-k\right) \left( n^{2}+\frac{%
2mn}{n-1}-4m\right) }}{n}\right] ^{\alpha }+\left( n-k\right) ^{1-\alpha } \\
&&\times \left[ \frac{2m\left( n-k\right) -\sqrt{mk\left( n-k\right) \left(
n^{2}+\frac{2mn}{n-1}-4m\right) }}{n}\right] ^{\alpha } \\
&=&\left( 2\left( n-1\right) \right) ^{\alpha }+\left( n-1\right) \left(
n-2\right) ^{\alpha }\text{, as }k=1\text{, }m=n\left( n-1\right) /2 \\
&=&s_{\alpha }\left( G\right) \text{, since }q_{1}=2\left( n-1\right) ,\text{
}q_{2}=\cdots =q_{n}=n-2.
\end{eqnarray*}%
This completes the proof of theorem.
\end{proof}

\begin{theorem}
Let $\alpha $ be a real number with $\alpha \neq 0,1$ and let $G$ be a
connected graph with $n\geq 3$ vertices and $t$ spanning trees and also let $%
t_{1}$ and $T$ be given by (1). For any real number $k\geq 0$,

i) if $G$ is bipartite, then 
\begin{equation}
s_{\alpha }\left( G\right) =\sigma _{\alpha }\left( G\right) >\left(
n-2\right) \left( nt\right) ^{\alpha /\left( n-1\right) }\left[ \frac{\left(
k+1\right) \left( nt\right) ^{\alpha /\left[ \left( k+1\right) \left(
n-1\right) \left( n-2\right) \right] }}{T^{\alpha /\left[ \left( k+1\right)
\left( n-2\right) \right] }}-k\right] +T^{\alpha }.
\end{equation}

ii) If $G$ is non-bipartite, then%
\begin{equation}
s_{\alpha }\left( G\right) >\left( n-1\right) \left( t_{1}\right) ^{\alpha
/n}\left[ \frac{\left( k+1\right) \left( t_{1}\right) ^{\alpha /\left[
\left( k+1\right) n\left( n-1\right) \right] }}{T^{\alpha /\left[ \left(
k+1\right) \left( n-1\right) \right] }}-k\right] +T^{\alpha }.
\end{equation}
\end{theorem}

\begin{proof}
By Lemmas 2.2--2.4, 2.10 and 2.11, the inequality (16) can be proved using
similar method of Theorem 3.4 in \cite{14}. We now only prove the inequality
(17).

Setting in Lemma 2.11 $a_{i}=q_{i}^{\alpha }$, $i=1,2,\ldots ,n$ and%
\begin{equation*}
p_{1}=\frac{k}{\left( k+1\right) n}\text{, }p_{i}=\frac{\left( k+1\right) n-k%
}{\left( k+1\right) n\left( n-1\right) }\text{, }i=2,3,\ldots ,n 
\end{equation*}%
we obtain%
\begin{eqnarray*}
&&\frac{kq_{1}^{\alpha }}{\left( k+1\right) n}+\frac{\left( k+1\right) n-k}{%
\left( k+1\right) n\left( n-1\right) }\sum\limits_{i=2}^{n}q_{i}^{\alpha
}-q_{1}^{\frac{k\alpha }{\left( k+1\right) n}}\prod\limits_{i=2}^{n}q_{i}^{%
\frac{\left( k+1\right) n-k}{\left( k+1\right) n\left( n-1\right) }\alpha }
\\
&\geq &\frac{k}{\left( k+1\right) n}\sum\limits_{i=1}^{n}q_{i}^{\alpha }-%
\frac{k}{k+1}\prod\limits_{i=1}^{n}q_{i}^{\alpha /n}.
\end{eqnarray*}%
Then, by Lemma 2.3, we have%
\begin{eqnarray*}
&&\frac{kq_{1}^{\alpha }}{\left( k+1\right) n}+\frac{\left( k+1\right) n-k}{%
\left( k+1\right) n\left( n-1\right) }\left( s_{\alpha }\left( G\right)
-q_{1}^{\alpha }\right) -q_{1}^{-\frac{\alpha }{\left( k+1\right) \left(
n-1\right) }}\left( t_{1}\right) ^{\frac{\left( k+1\right) n-k}{\left(
k+1\right) n\left( n-1\right) }\alpha } \\
&\geq &\frac{k}{\left( k+1\right) n}s_{\alpha }\left( G\right) -\frac{k}{k+1}%
\left( t_{1}\right) ^{\alpha /n},
\end{eqnarray*}%
i.e.,%
\begin{equation}
s_{\alpha }\left( G\right) \geq \left( n-1\right) \left[ \frac{\left(
k+1\right) \left( t_{1}\right) ^{\frac{\left( k+1\right) n-k}{\left(
k+1\right) n\left( n-1\right) }\alpha }}{q_{1}^{\frac{\alpha }{\left(
k+1\right) \left( n-1\right) }}}+\frac{q_{1}^{\alpha }}{n-1}-k\left(
t_{1}\right) ^{\alpha /n}\right] .
\end{equation}%
Let us consider the auxiliary function%
\begin{equation*}
f\left( x\right) =\frac{\left( k+1\right) \left( t_{1}\right) ^{\frac{\left(
k+1\right) n-k}{\left( k+1\right) n\left( n-1\right) }\alpha }}{x^{\frac{%
\alpha }{\left( k+1\right) \left( n-1\right) }}}+\frac{x^{\alpha }}{n-1}. 
\end{equation*}%
It is easy to see that $f\left( x\right) $ is increasing for $x>\left(
t_{1}\right) ^{1/n}$ whether $\alpha >0$ or $\alpha <0$. By Lemmas 2.3, 2.4
and Theorem 3.3 in \cite{4}, we have%
\begin{equation*}
q_{1}\geq T\geq \Delta _{1}+1>\Delta _{1}\geq \frac{2m}{n}\geq \left(
t_{1}\right) ^{1/n} 
\end{equation*}%
Therefore%
\begin{equation*}
f\left( x\right) \geq f\left( T\right) =\frac{\left( k+1\right) \left(
t_{1}\right) ^{\frac{\left( k+1\right) n-k}{\left( k+1\right) n\left(
n-1\right) }\alpha }}{T^{\frac{\alpha }{\left( k+1\right) \left( n-1\right) }%
}}+\frac{T^{\alpha }}{n-1}. 
\end{equation*}%
Combining this with (18) we get the inequality (17). Now we assume that the
equality holds in (17). Then all inequalities in the above arguments must be
equalities. Thus $q_{1}=T$ and $q_{1}=q_{2}=\cdots =q_{n}=\frac{2m}{n}$.
Thus we have that $q_{1}=\frac{2m}{n}\leq \Delta _{1}<\Delta _{1}+1\leq T$ \
which contradicts with the result in Lemma 2.4 \cite{4}. Hence (17) cannot
become an equality.
\end{proof}

\begin{remark}
By Lemmas 2.2 and 2.4, we have that $\mu _{1}=q_{1}\geq T\geq \Delta _{1}+1$
for bipartite graphs. Then from the proof of Theorem 3.4 in \cite{14}, one
can arrive at the bound (16) improves the bound of Theorem 3.4 in \cite{14}
for bipartite graphs.
\end{remark}

Taking $k=1$ in Theorem 3.4, we have the following result.

\begin{corollary}
Let $\alpha $ be a real number with $\alpha \neq 0,1$ and let $G$ be a
connected graph with $n\geq 3$ vertices and $t$ spanning trees and also let $%
t_{1}$ and $T$ be given by (1).

i) if $G$ is bipartite, then 
\begin{equation}
s_{\alpha }\left( G\right) =\sigma _{\alpha }\left( G\right) >\left(
n-2\right) \left( nt\right) ^{\alpha /\left( n-1\right) }\left[ \frac{%
2\left( nt\right) ^{\alpha /\left[ 2\left( n-1\right) \left( n-2\right) %
\right] }}{T^{\alpha /\left[ 2\left( n-2\right) \right] }}-1\right]
+T^{\alpha }.
\end{equation}

ii) If $G$ is non-bipartite, then%
\begin{equation}
s_{\alpha }\left( G\right) >\left( n-1\right) \left( t_{1}\right) ^{\alpha
/n}\left[ \frac{2\left( t_{1}\right) ^{\alpha /\left[ 2n\left( n-1\right) %
\right] }}{T^{\alpha /\left[ 2\left( n-1\right) \right] }}-1\right]
+T^{\alpha }.
\end{equation}
\end{corollary}

As in Remark 3.5, one can \ easily conclude that the bound (19) of Corollary
3.6 improves Corollary 3.5 in \cite{14}. Moreover, taking $\alpha =1/2$ in
Corollary 3.6, we have the following result.

\begin{corollary}
\cite{4} Let $G$ be a connected graph with $n\geq 3$ vertices and $t$
spanning trees and also let $t_{1}$ and $T$ be given by (1).

i) if $G$ is bipartite, then%
\begin{equation}
IE\left( G\right) =LEL\left( G\right) >\sqrt{T}+\left( n-2\right) \left(
nt\right) ^{1/\left[ 2\left( n-1\right) \right] }\left[ \frac{2\left(
nt\right) ^{1/\left[ 4\left( n-1\right) \left( n-2\right) \right] }}{T^{1/%
\left[ 4\left( n-2\right) \right] }}-1\right] .
\end{equation}

ii) if $G$ is non-bipartite, then%
\begin{equation}
IE\left( G\right) >\sqrt{T}+\left( n-1\right) \left( t_{1}\right) ^{1/\left(
2n\right) }\left[ \frac{2\left( t_{1}\right) ^{1/\left[ 4n\left( n-1\right) %
\right] }}{T^{1/\left[ 4\left( n-1\right) \right] }}-1\right] .
\end{equation}
\end{corollary}

\begin{theorem}
\bigskip Let $\alpha $ be a real number with $\alpha \neq 0,1$ and let $G$
be a connected graph with $n\geq 3$ vertices and $t$ spanning trees and also 
$t_{1}$ and $T$ be given by (1).

i) if $G$ is bipartite, then%
\begin{equation}
s_{\alpha }\left( G\right) =\sigma _{\alpha }\left( G\right) \geq T^{\alpha
}+\left( n-2\right) \left( \frac{nt}{T}\right) ^{\alpha /(n-2)}+\left(
\Delta _{2}^{\alpha /2}-\delta ^{\alpha /2}\right) ^{2}.
\end{equation}

ii) if $G$ is non-bipartite, then%
\begin{equation}
s_{\alpha }\left( G\right) >T^{\alpha }+\left( n-1\right) \left( \frac{t_{1}%
}{T}\right) ^{\alpha /(n-1)}+\left( \left( \Delta _{2}-1\right) ^{\alpha
/2}-\delta ^{\alpha /2}\right) ^{2}
\end{equation}%
where $\Delta _{2}$ and $\delta $ are the second maximum and the minimum
vertex degrees of the graph $G$, respectively.
\end{theorem}

\begin{proof}
Using Lemmas 2.2--2.4, 2.8, 2.9 and 2.12, one can prove inequality (23)
similar to the proof of Theorem 3.9 in \cite{14}. Here we only prove the
inequality (24).

By Lemma 2.12, we have%
\begin{equation*}
p\left( A_{p}-G_{p}\right) \geq \left( p-1\right) \left(
A_{p-1}-G_{p-1}\right) \geq \cdots \geq 2\left( A_{2}-G_{2}\right) 
\end{equation*}%
i.e.,%
\begin{equation}
A_{p}\geq G_{p}+\frac{2}{p}\left( \frac{a_{1+}a_{2}}{2}-\sqrt{a_{1}a_{2}}%
\right) =G_{p}+\frac{1}{p}\left( \sqrt{a_{1}}-\sqrt{a_{2}}\right) ^{2}
\end{equation}%
see, \cite{14}. Setting $p=n-1$, $\left( a_{1},a_{2},\ldots ,a_{n-1}\right)
=\left( q_{2}^{\alpha },q_{3}^{\alpha },\ldots ,q_{n}^{\alpha }\right) $ and 
$a_{1}=q_{2}^{\alpha }$, $a_{2}=q_{n}^{\alpha }$ in (25), we obtain%
\begin{equation*}
s_{\alpha }\left( G\right) =\dsum\limits_{i=1}^{n}q_{i}^{\alpha }\geq
q_{1}^{\alpha }+(n-1)\left( \dprod_{i=2}^{n}q_{i}\right) ^{\alpha
/(n-1)}+\left( q_{2}^{\alpha /2}-q_{n}^{\alpha /2}\right) ^{2}. 
\end{equation*}%
Considering Lemmas 2.3, 2.5 and 2.6, we have 
\begin{equation}
s_{\alpha }\left( G\right) =\dsum\limits_{i=1}^{n}q_{i}^{\alpha }\geq
q_{1}^{\alpha }+\left( n-1\right) \left( \frac{t_{1}}{q_{1}}\right) ^{\alpha
/\left( n-1\right) }+\left( \left( \Delta _{2}-1\right) ^{\alpha /2}-\delta
^{\alpha /2}\right) ^{2}.
\end{equation}%
Let us consider the auxiliary function%
\begin{equation*}
f\left( x\right) =x^{\alpha }+\left( n-1\right) \left( \frac{t_{1}}{x}%
\right) ^{\alpha /\left( n-1\right) }. 
\end{equation*}%
Note that $f\left( x\right) $ is increasing for $x>\left( t_{1}\right)
^{1/n} $ for both $\alpha >0$ and $\alpha <0$ \cite{3}. Then by Lemmas 2.3,
2.4 and Theorem 4.9 in \cite{3}, we have%
\begin{equation*}
f\left( x\right) \geq f\left( T\right) =T^{\alpha }+\left( n-1\right) \left( 
\frac{t_{1}}{T}\right) ^{\alpha /(n-1)}. 
\end{equation*}%
Combining this with Eq. (26), we get the inequality (24).

\begin{remark}
By Lemmas 2.2 and 2.4, we have that $\mu _{1}=q_{1}\geq T\geq \Delta _{1}+1$
for bipartite graphs. Then, from the proof of Theorem 3.9 in \cite{14}, one
can arrive at the bound (23) improves the bound of Theorem 3.9 in \cite{14}
for bipartite graphs. \ Moreover, it is clear that \ the results of Theorem
3.8 are better than the results of Theorem 4.9 in \cite{3}.
\end{remark}
\end{proof}

Taking $\alpha =1/2$ in Theorem 3.8, we get the following result on $IE$.

\begin{corollary}
Let $G$ be a connected graph with $n\geq 3$ vertices and $t$ spanning trees
and also let $t_{1}$ and $T$ be given by (1).

i) if $G$ is bipartite, then%
\begin{equation}
IE\left( G\right) =LEL\left( G\right) \geq \sqrt{T}+\left( n-2\right) \left( 
\frac{nt}{T}\right) ^{1/\left( 2(n-2)\right) }+\left( \Delta
_{2}^{1/4}-\delta ^{1/4}\right) ^{2}.
\end{equation}

ii) if $G$ is non-bipartite, then%
\begin{equation}
IE\left( G\right) >\sqrt{T}+\left( n-1\right) \left( \frac{t_{1}}{T}\right)
^{1/\left( 2(n-1)\right) }+\left( \left( \Delta _{2}-1\right) ^{1/4}-\delta
^{1/4}\right) ^{2}.
\end{equation}%
where $\Delta _{2}$ and $\delta $ are the second maximum and the minimum
vertex degrees of the graph $G$, respectively.
\end{corollary}

\begin{remark}
It is clear that the results of Corollary 3.10 improve the results of
Theorem 4.8 in \cite{3}.
\end{remark}

\begin{remark}
We finally note that, if we can establish a new lower bound such that $%
q_{1}\geq \beta \geq T$, then we can improve the results in Theorems 3.4 and
3.8.
\end{remark}

\textbf{Acknowledgments. }The authors are partially supported by T\"{U}B\.{I}%
TAK and the Office of Sel\c{c}uk University Research Project (BAP).

\end{document}